\documentclass[a4paper, 11pt, reqno]{amsart}

\usepackage{amsmath, amsthm, amsopn, amssymb}
\usepackage{xcolor}
\usepackage{graphicx}

\usepackage{geometry}
\newtheorem{theorem}{Theorem}[section]
\newtheorem{thm}{Theorem}[section]
\newtheorem{lemma}[theorem]{Lemma}

\newtheorem{conj}[theorem]{Conjecture}

\newtheorem{claim}[theorem]{Claim}

\theoremstyle{definition}

\theoremstyle{remark}
\newtheorem*{remark*}{Remark}

\newcommand{\bZ}{\ensuremath{\mathbb{Z}}}

\numberwithin{equation}{section}

\begin{document}

\title{Proof of the Brown-Erd\H{o}s-S\'os conjecture in groups}

\author{Rajko Nenadov}
\address{Rajko Nenadov, Department of Mathematics, ETH Zurich, Switzerland.}
\email{rajko.nenadov@math.ethz.ch}
\thanks{The first and the second author were supported in part by SNSF grant 200021-175573.}

\author{Benny Sudakov}
\address{Benny Sudakov, Department of Mathematics, ETH Zurich, Switzerland.}
\email{benjamin.sudakov@math.ethz.ch}

\author{Mykhaylo Tyomkyn}
\address{Mykhaylo Tyomkyn, Mathematical Institute, University of Oxford, Oxford OX2 6GG, UK}
\email{tyomkyn@maths.ox.ac.uk}
\thanks{The third author was supported in part by ERC Starting Grant 633509 and ERC Starting Grant 676632}


\date{\today}

\begin{abstract}
The conjecture of Brown, Erd\H{o}s and S\'os from 1973 states that, for any $k \ge 3$, if a $3$-uniform hypergraph $H$ with $n$ vertices does not contain a set of $k+3$ vertices spanning at least $k$ edges then it has $o(n^2)$ edges. The case $k=3$  of this conjecture is the celebrated $(6,3)$-theorem of Ruzsa and Szemer\'edi which implies Roth's theorem  on 
$3$-term arithmetic progressions in dense sets of integers. Solymosi observed that, in order to prove the conjecture, one can assume that $H$ consists of triples $(a, b, ab)$ of some finite quasigroup $\Gamma$. Since this problem remains open for all $k \geq 4$, he further proposed to study triple systems coming from finite groups. In this case he proved that the conjecture holds also for $k = 4$.  Here we completely resolve the Brown-Erd\H{o}s-S\'os conjecture for all finite groups and values of $k$. Moreover, we prove that the hypergraphs coming from groups contain sets of size $\Theta(\sqrt{k})$ which span $k$ edges. This is best possible and goes far beyond the conjecture.
\end{abstract}

\maketitle

\section{Introduction}

One of the main research directions in discrete mathematics concerns emergences of certain local sub-structures in objects of high density. Many classical results, such as Szemer\'edi's theorem on arithmetic progressions in subsets of integers of constant density or Tur\'an's theorem on the existence of complete graphs in very dense graphs, belong to this category of problems. In the study of hypergraphs, one of the most important open questions in this direction is the Brown-Erd\H{o}s-S\'os conjecture from 1973.

\begin{conj}[Brown-Erd\H{o}s-S\'os~\cite{BES}]\label{conj:BES0}
For any $c > 0$ and any integer $k$ there exists $n_0=n_0(c,k)$, such that every $3$-uniform hypergraph $H$ with $n \ge n_0$ vertices and at least $cn^2$ edges contains a subset of $k+3$ vertices which span at least $k$ edges. 
\end{conj}

Already the simplest case $k = 3$ of this conjecture, which is usually called the $(6,3)$-problem, had many interesting consequences. In particular, in the course of proving it Ruzsa and Szemer\'edi~\cite{RSz} used Szemer\'edi's regularity lemma to obtain an auxiliary result which is now known as the \emph{triangle-removal lemma}. This lemma and its extensions have many striking application in combinatorics, number theory and theoretical computer science. For example, it implies Roth's theorem~\cite{roth1953certain} on $3$-term arithmetic progressions in dense sets of integers and its stronger \emph{corner} version by Ajtai and Szemer\'edi \cite{ajtai75} (see \cite{solymosi2003note}). A removal lemma for larger complete graphs was later obtained by Erd\H{o}s, Frankl and R\"odl \cite{erdos1986asymptotic} in the course of extending the $(6,3)$-theorem of Ruzsa and Szemer\'edi to higher uniformities.
Deriving a hypergraph removal lemma was one of the driving forces behind development of the hypergraph regularity method (see, e.g., \cite{rodl2006density}), with one of the main applications in mind being a simpler proof of Szemer\'edi's theorem~\cite{szemeredi75kterm} which generalises Roth's theorem to arithmetic progressions of arbitrary length. 

Despite a lot of research in the last 40 years, the Brown-Erd\H{o}s-S\'os conjecture remains open for all values $k \ge 4$. The best upper bound on the number of vertices which are known to span $k$ edges is $k + 2 + \lceil \log k \rceil$, obtained by S\'ark\"ozy and Selkow \cite{sarkozy2004extension}.

It is not difficult to see that we may assume $H$ is linear, that is no two edges share more than one vertex. Indeed, if a pair of vertices in $H$ is shared by $k$ edges, then this already gives $k+2$ vertices spanning at least $k$ edges. Otherwise a simple greedy argument produces a linear subgraph $H'\subset H$ of size at least $(c/k)n^2$. Furthermore, by partitioning vertices of $H'$ at random into three parts we obtain a tripartite hypergraph $H''$ with $\frac{2}{9}(c/k)n^2$ edges. These hyperedges can be seen as entries of a partial $n \times n$ Latin square. Using a result of Evans~\cite{evans1960embedding} which states that every partial $n \times n$ Latin square can be embedded into a $2n \times 2n$ Latin square, Solymosi~\cite{Sol} observed that, by the previous, the Brown-Erd\H{o}s-S\'os conjecture can be phrased in terms of quasigroups\footnote{Recall that a finite set $\Gamma$ forms a quasigroup under a binary operation if it satisfies all group axioms except associativity, or, combinatorially, if its multiplication table forms a Latin square.}.

\begin{conj}\label{conj:BES2}
	For every integer $k \ge 3$ and $c > 0$, there exists $n_0 \in \mathbb{N}$ such that if $\Gamma$ is a finite quasigroup with $|\Gamma| \ge n_0$, then for every set $S$ of triples of the form $(a, b, ab) \in \Gamma^3$ with $|S| \ge c |\Gamma|^2$ there exists a subset $T \subseteq \Gamma$ of $k + 3$ elements which spans at least $k$ triples from $S$, that is, at least $k$ triples from $S$ belong to $T^3$.
\end{conj}
\begin{remark*}
Without loss of generality, here and in the rest of the paper, we assume that every triple (as a set) appears in $S$ only once. Moreover for every such triple we fix some ordering $(a,b, ab)$ in which the third element is the product of the first two.
\end{remark*}

As a step towards understanding this conjecture, Solymosi \cite{Sol} suggested to consider the case where $\Gamma$ is a \emph{group}. In particular, he showed that the
Brown-Erd\H{o}s-S\'os conjecture for groups holds also when $k = 4$. Very recently, while we were completing this paper, Solymosi and Wong~\cite{SW} made a further step towards  Brown-Erd\H{o}s-S\'os conjecture in groups. They proved that for any group and every set of quadratically many triples $(a,b,ab)$  there are infinitely many values of $k$ such that there is a set of size $(3/4+o(1))k$ spanning at least $k$ triples. In their result the value of $k$ can not be chosen in advance and depends on the group and more importantly on the set of triples. 

In this paper we completely resolve Brown-Erd\H{o}s-S\'os  problem in groups for all values of $k$. Unlike Conjecture \ref{conj:BES2} which, if true, would be optimal, we show that in the case of groups there are already sets of size $O(\sqrt{k})$ spaning $k$ triples. 

\begin{theorem} \label{thm:main}
	For every integer $k \ge 3$ and $c > 0$, there exists $n_0 \in \mathbb{N}$ such that if $\Gamma$ is a finite group with $|\Gamma| \ge n_0$, then for every set $S$ of triples of the form $(a, b, ab) \in \Gamma^3$ with $|S| \ge c |\Gamma|^2$ there exists a subset of $\Gamma$ of size at most
	$$
		\min\left\{k + 3, 8 \sqrt{k}\right\}
	$$
	which spans at least $k$ triples from $S$.
\end{theorem}

\noindent
Note that, since our hypergraphs are linear, the bound of $\Theta(\sqrt{k})$ is tight up to a constant factor. Interestingly, as $8\sqrt{k}$ does not depend on $c$ we have that triple systems coming from groups are much denser locally than globally. 

\medskip
\paragraph{\textbf{Note added in proof.}} After this paper was written we learned that Theorem \ref{thm:main} was proved independently by J. Long \cite{Long}, and a weaker result (where constant in front of $\sqrt{k}$ depends on density $c$ of triples) was obtained independently by Wong \cite{Wong}.

\section{Proof of Theorem \ref{thm:main}}

In the proof of Theorem \ref{thm:main} we utilise two classical theorems in additive combinatorics: the density version of the Gallai-Witt theorem \cite{furstenberg1978ergodic,gowers2007hypergraph,rodl2006density} (also known as the multidimensional Szemer\'edi's theorem) and the multidimensional density Hales-Jewett theorem \cite{dodos2014simple,furstenberg1991density,polymath2012new}. Let us recall them here, starting with the former. 

\begin{theorem} \label{thm:multi_Szm}
	Let $d$ be a positive integer, $R$ be a finite subset of $\mathbb{N}^d$, and $c > 0$. If $n \ge n_0(d, R, c)$ is sufficiently large, then every subset $C \subseteq [n]^d$ of size $|C| \ge c n^d$ contains a homothetic copy of $R$, that is there exist $s \in [n]^d$ and an integer $t \ge 1$ such that $s + t R \subseteq C$.
\end{theorem}

Let $m, d$ and $z$ be integers, with $d \le m$. A \emph{$d$-dimensional combinatorial subspace} of a cube $[z]^m$ is defined as follows: partition the ground set $[m]$ into
$z + d$ sets $X_1, \ldots, X_z, W_1, \ldots, W_d$ such that $W_1, \ldots, W_d$ are non-empty; the subspace consists of all sequences $x = (x_1, \ldots, x_m) \in [z]^m$ such that $x_i = j$ whenever $i \in X_j$ and $x$ is constant on each set $W_j$, that is if $i, i' \in W_j$ then $x_i = x_{i'}$. There is an obvious isomorphism between $[z]^d$ and any $d$-dimensional combinatorial subspace: the sequence $a = (a_1, \ldots , a_d)$ is sent to the sequence $x$ such that $x_i = j$ whenever $i \in X_j$ and $x_i = a_j$ whenever $i \in W_j$. With this notion at hand, we are ready to state the multidimensional density Hales-Jewett theorem.

\begin{theorem} \label{thm:HJ}
	For every $c > 0$ and every pair of integers $z$ and $d$ there exists a positive integer $MDHJ(z, d, c)$ such that, for $m \ge MDHJ(z, d, c$), every subset $C \subseteq [z]^m$ of size $|C| \ge c z^m$ contains a $d$-dimensional combinatorial subspace of $[z]^m$.
\end{theorem}

We now prove Theorem \ref{thm:main}. 

\begin{proof}[Proof of Theorem \ref{thm:main}]
	Let $\Gamma$ be a finite group with $|\Gamma| > n_0$, for some sufficiently large $n_0$. Let $G \subseteq \Gamma$ be an arbitrary subgroup of $\Gamma$ . Recall that the sets of both left and right cosets of $G$ partition the elements of the group $\Gamma$. Therefore direct product of such cosets $\ell G \times G r$ partitions $\Gamma \times \Gamma$ into sets of size $|G|^2$. 
	Thus, by averaging there exist $\ell, r \in \Gamma$ such that the set
	$$
	S_{\ell, r} = S \cap \{ (a, b, a b) \colon a \in \ell G, \; b \in G r \}
	$$
	is of size at least $|S_{\ell, r}| \ge c |G|^2$.	
	Let
	$$
	S' = \{ (a, b, a b) \colon a, b \in G\; \text{ such that }\; (\ell a, b r, \ell a  b  r) \in S\},
	$$
	and note that $|S'| = |S_{\ell, r}|$. Crucially, for any sets $A, B \subseteq G$ and $P \subseteq A B$ such that 
	\begin{equation} \label{eq:ABP}
	|S' \cap A \times B \times P| \ge k,
	\end{equation}
	we have
	$$
	|S \cap \ell A \times B r \times \ell  P  r| \ge k.
	$$
	Therefore, to prove the theorem it suffices to find such sets $A, B$ and $P$ in $G$ with $|A| + |B| + |P| \le \min\{k + 3, 8 \sqrt{k}\}$.

	By an observation of Erd\H{o}s and Straus~\cite{erdos1976abelian}, $\Gamma$ contains an abelian subgroup $\Gamma'$ with $|\Gamma'| \ge 0.9 \log |\Gamma|$. A better (and tight) estimate on the size of a largest abelian subgroup was obtained by Pyber~\cite{Py}, however this results relies on the classification of finite simple groups and for our purposes a much more elementary result of Erd\H{o}s and Straus suffices. In fact, any estimate which allows us to assume that $\Gamma'$ is sufficiently large, provided $\Gamma$ is large, would do as well.
	
	Let $K$ and $m$ be sufficiently large constants ($m$ will depend on $K$) which we choose later. From the fundamental theorem of finite abelian groups we have that $\Gamma'$ is isomorphic to a direct sum of the form
	$$
		\bigoplus_{i = 1}^h \bZ_{q_i}^{m_i},
	$$ 
	where all $q_i$'s are distinct. Therefore, by choosing $n_0$ to be large enough we can assume that there exists some $i \in [h]$ such that either $q_i>K$ or  $m_i>m$.
	For brevity let us call $q_i=q$ and $m_i=m$. By the above discussion, in the first case we can reduce problem to $\bZ_q$ and in the second to $\bZ_q^m$. In both cases, for the rest of the proof we switch to additive notation.

	\medskip
	\paragraph{\textbf{Case 1: $q \ge K$}} 
	Let $C \subseteq [q]^2$ be a subset consisting of all $(a,b) \in \bZ_q^2$ such that $(a, b, a+b) \in S'$, and note that $|C| \ge c q^2$. Choose $K$ sufficiently large, so that we can apply Theorem \ref{thm:multi_Szm} to find $s = (s_1, s_2) \in \bZ^2$ and some  positive integer $t$ such that $s + tR \subset C$, where $R = [k]^2$.

	Let $A = \{s_1 + t i \colon i \in [h]\}$ and $B = \{s_2 + t j \colon j \in [h]\}$, for $h = \lceil \sqrt{k} \rceil$. Note that for every $a \in A$ and $b \in B$ we have $(a, b) \in C$ and thus $(a, b, a+ b) \in S'$. Therefore, by the choice of $h$, the sets $A, B$ and $P = A + B$ satisfy \eqref{eq:ABP}. 
	As $A+B \subset \{s_1+s_2+t\ell: \ell \in [2,2h]\}$, we have $|A + B| < 2h$. Thus $|A| + |B| + |P| < 4 h \le 8 \sqrt{k}$, with room to spare.

	We apply a similar approach to find sets with the sum of sizes at most $k + 3$. For that choose $A = \{s_1 + ti \colon i \in [h]\}$, this time with $h = \lceil k/2 \rceil$, and $B = \{s_2 + t, s_2 + 2t\}$. If $k$ is even set $P = A + B$, and otherwise $P = (A + B) \setminus \{s_1 + ht + s_2 + 2t\}$. A routine check shows that in both cases the obtained sets satisfy \eqref{eq:ABP} and $|A| + |B| + |P| = k + 3$.

	\medskip
	\paragraph{\textbf{Case 2: $q < K$}.} 
	Choose $m$ to be sufficiently large so that we can apply density Hales-Jewett theorem (Theorem~\ref{thm:HJ}) with $c$, $d=k$ and $z=K^2$.
	Our aim is to show there exists a $k$-dimensional vector space $W \subseteq Z_q^m$ and $\hat a, \hat b \in Z_q^m$ such that for every $a \in \hat a + W$ and $b \in \hat b + W$ we have $(a, b, a+b) \in S'$. Before we prove that such $W$ and $\hat a, \hat b$ exist, let us first show how it implies the existence of desired sets $A, B$ and $P \subseteq A + B$.

	Let $u_1, \ldots, u_{k}$ be an arbitrary basis of $W$. Let $d \in \mathbb{N}_0$ be the largest integer such that $q^{2d} \le k$, and then let $t \in \mathbb{N}$ be the smallest integer such that $t^2 q^{2d} \ge k$. In particular, we have $1 \le t < q$. Set 
	$$
		W' = \big\{\lambda_1 u_1 + \ldots + \lambda_{d+1} u_{d+1} \colon \lambda_1, \ldots, \lambda_d \in [0, q-1], \lambda_{d+1} \in [0, t - 1] \big\}.
	$$
	For $A = \hat a + W'$ and $B = \hat b + W'$ we have $(a, b, a + b) \in S'$ for every $a \in A$ and $b \in B$, thus the choice of $t$ and $d$ implies that the sets $A, B$ and $P = A + B$ satisfy \eqref{eq:ABP}. As $A + B$ is of size at most $2t q^d$, we have $|A| + |B| + |P| \le 4 t q^d$. If $t = 1$ then $q^{2d} = k$, thus $4 t q^d \le 4 \sqrt{k}$. 
	Otherwise, for $t \ge 2$ we have	
	$$
		(t/2)^2 q^{2d} \le (t - 1)^2 q^{2d} < k,
	$$
	which implies $4 t q^d < 8 \sqrt{k}$. In either case, we have $|A| + |B| + |P| \le 8 \sqrt{k}$, as desired.

	As in Case 1, a similar approach is also used to find a subsets with the sum of sizes $k+3$. Let $h = \lfloor k / (2q - 1) \rfloor$, $t = \lceil (k - h (2q - 1))/2 \rceil$ and set
	$$
		A = \hat a + \{ \lambda u_i \colon \lambda \in [0, q - 1] \text{ and } i \le h\} \cup \{ \lambda u_{h+1} \colon \lambda \in [0, t - 1]\}
	$$
	and 
	$$
		B = \hat b + \{0, u_1, \ldots, u_h, u_{h+1}\}.
	$$
	Note that $|A| = h(q-1) + t$ and $|B| = h + 2$. Again, for $a \in A$ and $b \in B$ we have $(a, b, a + b) \in S'$. This time we do not take $P = A + B$, which would be too large, but only a subset of it, namely
	$$
		P = \hat a + \hat b + \{ \lambda u_i \colon \lambda \in [0, q - 1] \text{ and } i \le h\} \cup \{\lambda u_{h+1} \colon \lambda \in [0, t] \}.		
	$$
	Moreover, if $k - h(2q - 1)$ is not even then 
	$P := P \setminus \{ t u_{h+1} \}$. Note that $|P|=h(q-1) + t+1$ when $k - h(2q - 1)$ is even and $|P|=h(q-1) + t$ otherwise. Thus, it is easy to verify that $|A| + |B| + |P| = k + 3$. Let us briefly check that \eqref{eq:ABP} is also satisfied. We do this only in the case $k - h(2q - 1)$ is even. The other case is done analogously. First, for every $i \in \{1, \ldots, h\}$ and every 
	$$
		a \in \hat a + \{\lambda u_i \colon \lambda \in [0, q - 1]\} \quad \text{ and } \quad b \in \{\hat b, \hat b + u_i\}
	$$
	we have $a + b \in P$. Overall this amounts to $h \cdot 2 q - (h-1)$ triples in $S'$. Similarly, for every 
	$$
		a \in \hat a + \{\lambda u_{h+1} \colon \lambda \in [0, t-1]\} \quad \text{ and } \quad b \in \{\hat b, \hat b + u_i\}
	$$
	we again have $a + b \in P$, which contributes additional $2t - 1$ triples (notice that we have already counted $\hat a + \hat b$ in the previous step). Using $2t = k - 2hq + h$ we conclude that this amounts to $k$ triples in total.
	
	It remains to show, using Theorem~\ref{thm:HJ}, that a desired $k$-dimensional vector subspace $W \subseteq Z_q^m$ and elements $\hat a, \hat b$ exist. Consider a subset $C \subseteq V^m$, where $V = \bZ_q \times \bZ_q$, which contains an element 
	$$
		((a_1,b_1), \ldots, (a_m, b_m))
	$$
	if and only if $(a, b, a + b) \in S'$ for $a = (a_1, \ldots, a_m)$ and $b = (b_1, \ldots, b_m)$. As $|C| \ge c (q^m)^2 = c |V|^m$ and $m$ is sufficiently large, by Theorem \ref{thm:HJ} the set $C$ contains a $k$-dimensional combinatorial subspace. Let $\{X_{(e_1, e_2)}\}_{(e_1, e_2) \in V}$ and $W_1, \ldots, W_{k}$ be the partition of $[m]$ corresponding to this subspace. Define $\hat a \in \bZ_q^m$ by setting $\hat a_i = j$ for every $i \in \bigcup_{e_2 \in \bZ_q} X_{(j, e_2)}$ and, similarly, $\hat b_i = j$ for every $i \in \bigcup_{e_1 \in \bZ_q} X_{(e_1, j)}$. For all $i \in W_1 \cup \ldots \cup W_{k}$ set $\hat a_i = \hat b_i = 0$. Furthermore, let $u_1, \ldots, u_k \in \bZ_q^m$ be vectors defined as $(u_i)_j = 1$ for $j \in W_i$ and $(u_i)_j = 0$ otherwise, for $i \in [k]$. It is clear that they are independent in $\bZ_q^m$ and therefore span a $k$-dimensional vector subspace, which we denote by $W$.

	Let us briefly check that the obtained $W$ and $\hat a, \hat b$ have the desired property. Consider some $a = (a_1, \ldots, a_m) \in \hat a + W$ and $b  = (b_1, \ldots, b_m) \in \hat b + W$. Then $((a_1, b_1), \ldots, (a_m, b_m))$ belongs to a $k$-dimensional combinatorial subspace of $V^m$ given by the partition $\{X_{(e_1, e_2)}\}_{(e_1, e_2) \in V}$ and $W_1, \ldots, W_{k}$. As this combinatorial subspace lies in $C$, from the definition of $C$ we conclude $(a, b, a + b) \in S'$.
\end{proof}

\section{Concluding remarks}

Theorem \ref{thm:main} shows that triples coming from groups contain much denser subsets than conjectured. 
Determining the best possible constant $C$ in the $C\sqrt{k}$-term of Theorem~\ref{thm:main} remains an interesting problem. We were able to do it for cyclic groups $\bZ_n$, where we obtain $C=\sqrt{12}$. We believe that the proof, which is presented in the Appendix, is interesting in its own right as it establishes a correspondence between the Brown-Erd\H{o}s-S\'os conjecture for $\bZ_n$ and the following discrete isoperimetric problem. Recall that the \emph{edge-boundary} of a vertex set $S$ in a graph is defined as $\partial_e(S):=e(S,\overline{S})$, that is the number of edges leaving $S$. The edge-isoperimetric problem for a graph $G$ (that may be infinite) and an integer $k$ asks to find the minimum edge-boundary of vertex sets of size $k$ in $G$. 

Here we are particularly concerned with $G$ being the two-dimensional triangular lattice $T$, where the above question was answered by Harper~\cite[Theorem 7.2]{harper_2004}.
\begin{thm}[\cite{harper_2004}]\label{thm:harper}
There exists a nested family of vertex sets in $T$, $\mathcal{S}= (S_k)_{k\in \mathbb{N}}$, with $|S_k|=k$, such that each $S_k$ minimizes the edge-boundary over all sets of size $k$. The family $\mathcal{S}$ contains all balls in the lattice metric of $T$, i.e. regular hexagons.
\end{thm}  
The following figure visualizes $\mathcal{S}$: the set $S_k$ consists of all vertices labeled $1$ through $k$.

\medskip
\begin{center}
\includegraphics{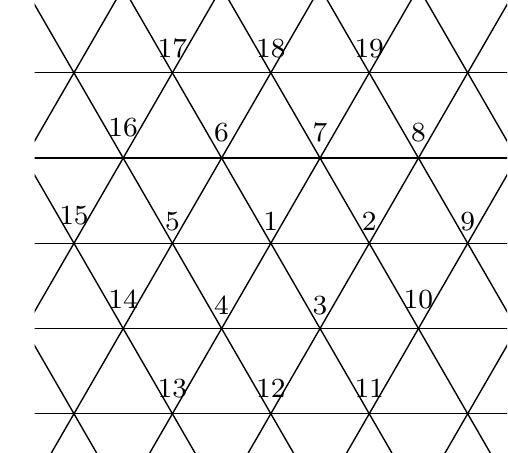}  
\end{center}

\medskip

The proof in~\cite{harper_2004} is not elementary, as it makes use of some powerful abstract tools that can be applied to various other isoperimetric problems. It might be therefore of independent interest that as a by-product of determining the correct constant in Theorem~\ref{thm:main} for $\bZ_n$, we, somewhat unexpectedly, obtain a short elementary proof of Harper's theorem. Recently Angel, Benjamini and Horesh~\cite[Theorem 2.4]{ABH} proved an edge-isoperimetric inequality for planar triangulations which generalises Harper's theorem. Our result can be viewed as an extension of Theorem~\ref{thm:harper} in a different direction, as we determine the minimum number of axis-parallel lines occupied by any set of $k$ points in $T$.

Finally, it would be interesting to determine the correct constant also in the $\mathbb{Z}_q^m$-case as this would in turn yield the optimal constant in Theorem~\ref{thm:main}.

\section*{Acknowledgement}
We would like to thank Asaf Shapira for arranging the third author's visit to ETH Zurich where the present research was conducted.

\bibliographystyle{abbrv}
\bibliography{references}

\section*{Appendix}

Here we shall determine the sharp constant in Theorem~\ref{thm:main} for cyclic groups. For a set of points $P \subseteq \mathbb{Z}^2$ define 
$$g(P):=|\{x\colon(x,y)\in P\}|+|\{y\colon(x,y)\in P\}|+|\{z \colon(x,y)\in P, x+y=z\}|.
$$ 
That is, $g(P)$ measures the total number of rows, columns and lines of form $x+y=c$ (`diagonals' for short -- note that we completely ignore the diagonals of the type $x=y+c$) occupied by points in $P$.  Let
$$
	g(k)=\min_{|P|=k}g(P).
$$ 
The following lemma shows that $g(k)$ precisely determines the size of a smallest subset which is guaranteed to span at least $k$ edges, in the case where $\Gamma = \bZ_n$.
\begin{lemma} \label{lemma:Zn}
	For every integer $k \ge 3$ and $c > 0$, there exists $n_0 \in \mathbb{N}$ such that if $n \ge n_0$ then for every set $S$ of triples of the form $(a, b, a + b) \in \bZ_n^3$ with $|S| \ge c n^2$, there exists a subset of $\bZ_n$ of size $g(k)$ which spans at least $k$ triples from $S$. Moreover, there exists a set $S$ of $n^2 / 64$ triples in which every subset of size $g(k)-1$ spans less than $k$ edges.
\end{lemma}
\begin{proof}
Let $S$ be a given set of triples, and let $C \subseteq [n]^2$ be the set of points containing all $(a, b)$ such that $(a, b, a + b) \in S$. In other words, each point in $C$ corresponds to a triple from $S$. Next, let $P \subseteq [m]^2$ be a set of $k$ points such that $g(P) = g(k)$, where $m \in \mathbb{N}$, clearly, depends only on $k$. By Theorem \ref{thm:multi_Szm}, $C$ contains a homothetic copy of $P$, that is there exist some $s \in [n]^2$ and an integer $t$ such that $s + t P \subseteq C$. We claim that the $k$ edges corresponding to points in $s + tP$ span at most $g(k)$ elements of $\bZ_n$. This easily follows from the observation that $g(s + tP) = g(P) = g(k)$ and the number of different elements $a + b$ where $(a,b) \in s + tP$ is at most the size of the set
$$
	Q=\{z \colon (a, b) \in s + tP, a + b = z\}
$$
where the addition is done in $\bZ$ instead of $\bZ_n$ (hence, $|Q \mod n|$ cam be strictly smaller than $|Q|$).

Let us now exhibit a set of triples $S$ which shows the optimality of $g(k)$. Consider two intervals in $\mathbb{Z}_n$: $A=[n/8, 2n/8 - 1]$ and $B=[2n/8, 3n/8 - 1]$. Then $A+B=[3n/8, 5n/8 - 2]$ and, in particular, $A$,$B$ and $A+B$ are disjoint. Let $S$ the set of all triples $(a, b, a + b)$  where $a\in A$ and $b\in B$, and note that $|S|=n^2/64$. Identify each triple $(a, b, a + b) \in S$ with a point $(a,b) \in A\times B \subseteq \mathbb{Z}^2$. Then for a set of $k$ triples in $S$, corresponding to a set $P$ of $k$ points in $\mathbb{Z}^2$, the involved vertices in $A$, $B$ and $A+B$ correspond to rows, columns and diagonals occupied by $P$, respectively, owing to disjointness of these sets. Hence, any set of $k$ edges necessarily span at least $g(k)$ elements.
\end{proof}

The following theorem determines the growth rate of $g(k)$ and, by Lemma \ref{lemma:Zn}, tight bounds on the size of a smallest set which spans $k$ edges in additive triples coming from $\Gamma = \bZ_n$.

\begin{thm} \label{thm:gk}
	$g(k) = (1 + o(1)) \sqrt{12 k}$.
\end{thm}

We prove Theorem \ref{thm:gk} by considering a dual problem: given an integer $m$, what is the size of a largest set of points $P \subseteq \bZ^2$ such that $g(P) = m$?
Note that $g$ and $h$ are inverse functions, in the sense that if $g(k)=m$ and $g(k')=m+1$ then $k\leq h(m)<k'$.

To determine the growth rate of $h$, consider some fixed sets $A,B\subseteq \mathbb{Z}$, with $|A|=a, |B|=b$, and let $h(A,B,\ell)$ be the largest number of points in $A\times B$ occupying at most $\ell$ diagonals. The following lemma, which is the heart of the proof of Theorem \ref{thm:gk}, shows that we can assume $A$ and $B$ to be intervals. 

\begin{lemma}\label{lem:equidist}
$h(A,B,\ell)\leq h([a],[b],\ell)$.
\end{lemma}

\begin{proof}
We apply induction on $\ell$; for $\ell=0$ there is nothing to prove. Suppose that the statement holds for $\ell-1$, for some $\ell \ge 1$. Let $C$ be a set of $\ell$ diagonals (recall that we only consider diagonals of the form $x + y = z$, for some integer $z$) and suppose, towards a contradiction, that $P = C \cap (A\times B)$ satisfies $|P| > h([a],[b],\ell)$. 

Let $A = \{x_1, \ldots, x_a\}$ and $B = \{y_1, \ldots, y_b\}$ such that $x_i < x_{i+1}$ and $y_i < y_{i+1}$. For each point $(x_i, y_j) \in A \times B$ consider the point $(i, j) \in [a] \times [b]$, and let
$$
D=\{(i,j): (x_i,y_j) \in P\} \subseteq [a]\times [b]
$$
be the set of all such points. Let $P' \subseteq [a] \times [b]$ be a set of points which is a certificate for $h([a], [b], \ell - 1)$, and recall that $P'$ is a union of $\ell - 1$ diagonals intersecting $[a] \times [b]$. In particular, if $(i, j) \notin P'$ then $(x, y) \notin P'$ for every $x \in [a]$ and $y \in [b]$ such that $x + y = i + j$. Note that
$$
	|D| = |P| > h([a],[b],\ell) \geq h([a],[b],\ell-1) = |P'|,
$$ 
thus there exists some $(i,j) \in D \setminus P'$. Let $P_1=\{(x,y)\in A \times B \colon x + y = x_i + y_j\} \subseteq P$ and $P_2 = P\setminus P_1$. In other words, $P_1$ consists of all the points of $P$ which lie on the same diagonal as $(x_i, y_j)$, and $P_2$ are all the points which remain after removing this diagonal.

\begin{claim}
 $$
 	|P_1| \leq |\{(x,y)\in [a]\times [b]: x + y = i + j\}|=:|T|.
 $$ 
\end{claim}
\begin{proof}
Let
$$
	D_1=\{(u,w):(x_u,y_w) \in P_1\}.
$$
Note that $(i, j) \in D_1\cap T$ and define $D_1^-:=\{(x,y)\in D_1: x < x_i\}$. For all $x,y$ with $x+y=x_i+y_j$ we have at most one such $y$ for every $x$ and vice versa, and if $x<x_i$ then $y>y_j$. Since there are at most $i-1$ values $x<x_i$ and at most $b-j$ values $y<y_j$, we get
$$|D_1^-|\leq \min\{i - 1, b - j\}= |T\cap ([i-1]\times [b])|.
$$
Analogously, for $D_2^+:=\{(x,y)\in D_1: x > x_i\}$ we obtain
$$|D_2^-|\leq |T\cap ([i+1]\times [b])|,
$$
and the claim follows by addition.
\end{proof}

By the induction hypothesis we have
$$ 
	|P_2| \leq h(A,B,\ell-1) \leq h([a],[b],\ell-1)=|P'|,
$$
and since $P' \cap T = \emptyset$, by an earlier observation, we obtain 
$$ 
	h([a],[b],\ell) \geq |P'| + |T| \geq |P_2|+|P_1|=|P| > h([a],[b],\ell),
$$
thus a contradiction. This proves the induction step, and the statement follows.
\end{proof}

The previous lemma reduces the problem of estimating $h(m)$ to finding integers $a, b, \ell$, such that $a + b + \ell = m$, which maximize $h([a], [b], \ell)$. 

\begin{proof}[Proof of Theorem \ref{thm:gk}]
To prove the theorem it suffices to determine the growth rate of $h(m)$.
By Lemma \ref{lem:equidist}, for any integer $m$ we have that 
$$
	h(m) = \max\left\{ h([a], [b], \ell) \colon a, b, \ell \in \mathbb{N} \text{ such that } a + b + \ell = m \right\}.
$$
For brevity, we write $h([a], [b], \ell) =: h(a, b, \ell)$.

\begin{claim}\label{cl:makefatter}
$h(m)$ is realised by $h(a,b,\ell)$, where $\lfloor m/3 \rfloor \leq a,b,\ell \leq \lceil m/3 \rceil$.
\end{claim}
\begin{proof}
Let $d_i(a,b)$ be the size of the $i$-th largest intersection of a diagonal with $[a]\times [b]$. Then, for $a\leq b$, we have 
$$d_1=\dots =d_{b-a+1}=a, d_{b-a+2}=d_{b-a+3}=a-1, \dots , d_{a+b-2}=d_{a+b-1}=1.$$
A term by term comparison of $h(a,b,\ell)=\sum_{i=1}^\ell d_i(a,b)$ and $h(a+1,b-1,\ell)=\sum_{i=1}^\ell d_i(a+1,b-1)$ shows that to achieve $h(m)$ we must have $|b-a|\leq 1$. Similarly, by comparing  
$h(a,a,\ell)$ with $h(a+1,a+1,\ell-2)$ and $h(a-1,a-1,\ell+2)$, and $h(a,a+1,\ell)$ with $h(a+1,a+2,\ell-2)$ and $h(a-1,a,\ell+2)$   
(alternatively, it is not difficult to see that $h(a,b,c)=h(a,c,b)=h(b,c,a)$ always holds) we obtain that $h(m)$ is realised when $a,b,c$ are within $1$ of each other. We omit the straightforward calculations.
\end{proof}
By the above discussion
$$
h(a,a,a)=a+2(a-1)+\dots+2(a/2)+o(a^2) = 2\binom a 2 - 2\binom {a/2}{2}+o(a^2) =(3/4)a^2+o(a^2).
$$
Therefore
$$h(m)=\frac{3}{4}(m/3)^2+o(m^2)=\frac{m^2}{12}+o(m^2).
$$
Inverting the function yields $g(k)=(1+o(1))\sqrt{12k}$, completing the proof of Theorem~\ref{thm:gk}.
\end{proof}

Note that as a corollary of Theorem~\ref{thm:gk} we immediately obtain an asymptotic version of Harper's theorem (Theorem~\ref{thm:harper}). Too see this, observe that the triangular lattice $T$ is isomorphic to the square lattice with all the diagonals $x+y=c$ `drawn in' (formally: the Cayley graph on $\bZ^2$ generated by $(1,0)$, $(0,1)$ and $(1,-1)$), where in the latter the edge-boundary of a set $P$ satisfies 
\begin{equation}\label{eq:lattices}
\partial_e(P) \geq 2g(P)\geq 2g(|P|).  
\end{equation}
Thus, by Theorem~\ref{thm:gk} we obtain $\partial_e(P)\geq 4(1+o(1))\sqrt{3}|P|$, which asymptotically matches the edge-boundary of the regular hexagons. 

To derive Theorem~\ref{thm:harper} in full from here, note that in the course of the proof we determine $g(k)$ via $h(m)$ precisely. Since the extremal sets claimed in Theorem~\ref{thm:harper} are also extremal sets for $g(k)$ (as the corresponding unions of rows, columns and diagonals are extremal for $h(m)$), Theorem~\ref{thm:harper} follows. 

Finally, note that Harper's theorem does not claim a complete classification of extremal sets for the edge-isoperimetric problem on $T$. In fact, for most values for $k$ it is easy to see that even up to isometry there is more than one extremal example. That said, the extremal examples \emph{are} unique for values of $k$ that are volumes of balls in $T$. This can be deduced from our argument as follows. If $P$ is extremal, by~\eqref{eq:lattices} it has to have no `gaps' (the intersection with each of the three axes has to be an interval), and be extremal for $g(k)$. However, the regular hexagon of radius $a$ in $T$ corresponds in $T'$ to the union of the $2a+1$ longest diagonals in $[2a+1]\times [2a+1]$, which, by a uniqueness analysis in Claim~\ref{cl:makefatter}, is the unique up to dilation extremal set for $h(3(2a+1))$ and therefore the unique gap-free extremal set for $g(k)$.
\end{document}